\documentclass{amsart}
\usepackage[T1]{fontenc}   

\usepackage{amsfonts}
\usepackage{ulem,dsfont,amsmath,amsthm,latexsym,amscd,bbm,amssymb,graphicx,paralist,epsfig,graphics,exscale}
\usepackage[small,nohug,heads=vee]{diagrams}
\usepackage{mathrsfs}
\usepackage{color,soul}

\begin{document}

\noindent                                             
\begin{picture}(150,36)                               
\put(5,20){\tiny{Accepted by}}                       
\put(5,7){\textbf{Topology Proceedings}}              
\put(0,0){\framebox(140,34){}}                        
\put(2,2){\framebox(136,30){}}                        
\end{picture}                                         
\vspace{0.5in}

\renewcommand{\bf}{\bfseries}
\renewcommand{\sc}{\scshape}
\vspace{0.5in}
\swapnumbers
\newtheorem{thm}{Theorem}[section]
\newtheorem*{3.4B}{\ref{3.4}B.~Theorem}
\newtheorem*{3.5A}{\ref{3.5}A.~Theorem}
\newtheorem*{3.5B}{\ref{3.5}B.~Corollary}
\newtheorem*{3.5C}{\ref{3.5}C.~Corollary}
\newtheorem{cor}[thm]{Corollary}
\newtheorem{lem}[thm]{Lemma}
\newtheorem{prop}[thm]{Proposition}

\newtheorem*{5.1}{\ref{5}.1.~Proposition}
\newtheorem*{5.2}{\ref{5}.2.~Proposition}
\newtheorem*{5.3A}{\ref{5.3}A.~Proposition}
\newtheorem*{5.3B}{\ref{5.3}B.~Proposition}

\theoremstyle{definition}
\newtheorem{se}[thm]{}
\newtheorem{exa}[thm]{Example}
\newtheorem{que}[thm]{Question}
\newtheorem{definition}[thm]{Definition}
\newtheorem*{example}{Example}
\newtheorem*{note*}{Note}
\newtheorem*{3.4A}{\ref{3.4}A.~Example}
\newtheorem{rem}[thm]{Remark}
\numberwithin{equation}{section}
\setcounter{section}{-1}

\title[\textsl{sncc}-Inheritance of pointwise a.p. in flows]%
{On \textsl{sncc}-inheritance of pointwise almost periodicity in flows}

\author[X.-P.~Dai]{Xiongping Dai}
\address{Department of Mathematics; Nanjing University;
Nanjing 210093; Peoples Republic of China}
\email{xpdai@nju.edu.cn}
\thanks{The author was supported in part by NSFC Grant \#12271245.}


%


\begin{abstract}
Let $H$ be a subnormal co-compact closed subgroup of a Hausdorff topological group $T$ and $X$ a compact Hausdorff space.
We prove the inheritance theorem: A point of $X$ is almost periodic (a.p.) for $T\curvearrowright X$ iff it is a.p. for $H\curvearrowright X$.
Moreover, if $T\curvearrowright X$ is minimal with $H\lhd T$, then $\mathscr{O}_H\colon X\rightarrow2^X$, ${x\mapsto\overline{Hx}}$ is a continuous mapping, and, $T\curvearrowright X/H$ is an a.p. nontrivial factor of $T\curvearrowright X$ iff $T\curvearrowright X\times T/H$ is not minimal.
\end{abstract}
\maketitle

\section{\bf Introduction}
Let $X$ be a compact Hausdorff space as our phase space and $T$ a Hausdorff topological group as our phase group, unless otherwise specified.
We say that $(T,X)$ is a \textit{flow}, denoted $\mathscr{X}$ or $T\curvearrowright X$,  if there exists a continuous map $T\times X\xrightarrow{(t,x)\mapsto tx}X$, called the phase mapping, such that $ex=x$ and $(st)x=s(tx)$ for all $x\in X$ and $s,t\in T$, where $e\in T$ is the identity element.
If $\mathscr{X}$ has some dynamical properties and $H$ is a closed subgroup of $T$, it is natural to ask whether or not the induced subflow $H\curvearrowright X$ necessarily inherits these dynamical properties.

We shall consider the pointwise almost periodic inheritance of subnormal co-compact closed subgroup action (see Thm.~\ref{3.1} and its applications Thm.~\ref{3.4}B, Thm.~\ref{6.5}, and Thm.~\ref{6.8}). In literature \ref{3.4}B is a classical theorem that gives us the first example of minimal distal non-equicontinuous metric flows. Its proofs available in the literature (e.g., \cite[Chap.~IV]{AGH} and \cite[Exa.~6.19.2]{E69}) are cumbersome and lengthy. Ours is however very concise using the inheritance theorem.

\section{\bf Basic notions and lemmas}\label{s1}
Let $\mathscr{X}$ be a flow.
By $2^T$ and $2^X$ we denote respectively the collections of nonempty compact subsets of $T$ and $X$, where $2^X$ is endowed with the Vietoris topology. By $\mathfrak{N}_e(T)$ and $\mathfrak{N}_x(X)$ we denote the systems of neighborhoods of $e$ in $T$ and of $x$ in $X$, respectively. Here $\{pt\}$ stands for a singleton.

\begin{se}\label{1.1}
Let $H$ be any subgroup of $T$. Then the quotient map $\rho\colon T\rightarrow T/H$ is surjective, where $T/H=\{tH\,|\,t\in T\}$ is the left-coset space. By
\begin{enumerate}
\item[] $\mathscr{O}_H\colon X\rightarrow 2^X$, $x\mapsto\overline{Hx}$
\end{enumerate}
it stands for the lower semi-continuous $H$-orbit-closure mapping.
\end{se}

\begin{se}\label{1.2}
A set $A$ in $T$ is \textit{syndetic}~\cite[Def.~2.02]{GH} if there exists $K\in2^T$ such that $T=K^{-1}A$, or equivalently, $Kt\cap A\not=\emptyset\ \forall t\in T$.
\end{se}

\begin{se}\label{1.3}
A point $x$ in $X$ is called \textit{almost periodic} (a.p.) for $\mathscr{X}$ if for every $U\in\mathfrak{N}_x(X)$
\begin{enumerate}
\item[] $N_T(x,U)=\{t\,|\,t\in T, tx\in U\}$
\end{enumerate}
is syndetic in $T$. If every point of $X$ is a.p., then $\mathscr{X}$ is said to be \textit{pointwise a.p.}
See \cite{NS, GH, G76, B79, A, De}.

Let $\mathscr{U}$ be the uniformity on $X$. If, for each $\alpha\in\mathscr{U}$, there exists a syndetic set $A$ in $T$ such that $Ax\subseteq\alpha[x]$ for all $x\in X$, then $\mathscr{X}$ is called an {\it a.p. flow}. Note that $\mathscr{X}$ is a.p. iff it is \textit{equicontinuous}; i.e., given $\varepsilon\in\mathscr{U}$ there exists $\delta\in\mathscr{U}$ such that $T\delta\subseteq\varepsilon$\,(cf.~\cite{GH, E69, B79, A, De}).
\end{se}

\begin{se}\label{1.4}
$\mathscr{X}$ is \textit{minimal} iff $\overline{Tx}=X$ for all $x\in X$. In fact, a point $x$ is a.p. for $\mathscr{X}$ iff $\overline{Tx}$ is a minimal set of $\mathscr{X}$ (see, e.g., \cite{NS, G46, GH}).
\end{se}

\begin{se}\label{1.5}
By $\mathcal {S}^{\textsl{ns}}(T)$ we denote the family of the normal syndetic closed subgroups of $T$.
\end{se}

\begin{se}\label{1.6}
A subgroup $H$ of $T$ is said to be \textit{co-compact} iff the quotient space $T/H$ is compact. In that case, $T$ is also called a compact extension of $H$.
We denote by $\mathcal{S}^{\textsl{ncc}}(T)$ the family of the normal co-compact closed subgroups of $T$.
\end{se}

\begin{lem}\label{1.7}
Let $T$ be an ``\textsl{w\!F}-group''\,(i.e., $T$ is a group with a topology under which $(s,t)\mapsto st^{-1}$ is separately continuous).\footnote{A group $K$ that is compact with multiplication separately continuous and inversion continuous is refereed to as an ``\textsl{F}-group'' in \cite[Rem.~2.5]{V77}. An locally compact\,(LC) Hausdorff \textsl{w\!F}-group is a topological group according to Ellis' Joint Continuity Theorem.} Let $H_1<H_2$ be closed subgroups of $T$; then:
\begin{enumerate}[(a)]
\item[$(\mathrm{a})$] If $H_1$ is syndetic in $T$, then it is syndetic in $H_2$ and $H_2$ is also syndetic in $T$\,(cf., e.g.,~\cite{GH, B79}).
\item[$(\mathrm{b})$] If $H_1$ is co-compact in $T$, then it is co-compact in $H_2$ and $H_2$ is also co-compact in $T$.
\item[$(\mathrm{c})$] If $T$ is a semi-topological group and $H_1$ is syndetic in $H_2$ and $H_2$ is syndetic in $T$, then $H_1$ is syndetic in $T$\,(cf., e.g.,~\cite{GH, B79}).
\end{enumerate}
\end{lem}

\begin{proof}
\item (a): Obvious for a closed subset of a compact set is compact.
\item (b): This follows from the following canonically determined CD of continuous mappings:
\begin{diagram}
&&T&&\\
&\ldTo~{\rho_2^{}}&&\rdTo~{\rho_1^{}}&\\
T/H_2&&\lTo&&T/H_1&\hookleftarrow&H_2/H_1
\end{diagram}
and $H_2/H_1$ is closed in $T/H_1$.
\item (c): Obvious. The proof is completed.
\end{proof}

Note that if $H$ is a syndetic closed subgroup of $T$, then there exists $K\in2^T$ which is mapped onto $T/H$ by $\rho\colon T\rightarrow T/H$, and so $H$ is co-compact. When $T$ is LC, then the converse is true\,(cf., e.g.,~\cite[2.01]{GH} or \cite[2.8.18]{B79}).
In fact, this easily follows from the openness of $\rho$ as follows:

\begin{lem}\label{1.8}
Let $H$ be a subgroup of a topological group $T$, where $T$ need not be Hausdorff. Then:
\begin{enumerate}[$(\mathrm{a})$]
\item[$(\mathrm{a})$] The quotient map $\rho\colon T\rightarrow T/H$ is an open surjective map\,(cf.~\cite[Thm.~5.17]{HR}).
\item[$(\mathrm{b})$] If $H$ is closed, then $T/H$ is Hausdorff. (Thus, if $s,t\in T$ cannot be separated by open sets, then $sH=tH$ in $T/H$. See \cite[Thm.~5.21]{HR}.)\footnote{Recall that a topological space $Y$ is Hausdorff iff the diagonal
$\Delta_Y=\{(y,y)\,|\,y\in Y\}$
is closed in $Y\times Y$.}
\item[$(\mathrm{c})$] If $\{t_iH\colon i\in I\}$ converges in $T/H$ to $tH$, then there exists a subnet $\xi\colon I^\prime\to I$ and a net $\{k_{i^\prime}\colon i^\prime\in I^\prime\}$ in $T$ which converges in $T$ to $t$ with $k_{i^\prime}H=t_{\xi(i^\prime)}H$ for all $i^\prime\in I^\prime$. In particular, there exists a net $\{h_{i^\prime}\colon i^\prime\in I^\prime\}$ in $H$ such that $t_{\xi(i^\prime)}=k_{i^\prime}h_{i^\prime}$ for all $i^\prime\in I^\prime$.
\end{enumerate}
\end{lem}

\begin{proof}
\item (a): If $U$ is open in $T$, then $UH=\bigcup_{s\in H}Us$ is open in $T$ such that $UH=\rho^{-1}\rho[U]$. Since $\rho$ is the quotient map, $\rho[U]$ is open in $T/H$.
\item (b): Since $\rho$ is an open map, $\rho\times\rho\colon T\times T\rightarrow T/H\times T/H$ is open and so is a quotient map. As
    \begin{enumerate}
    \item[] $(\rho\times\rho)^{-1}[\Delta_{T/H}]=\{(s,t)\in T\times T\,|\,s^{-1}t\in H\}$
    \end{enumerate}
    which is $\rho\times\rho$-saturated and which closed because $H$ is closed and $T$ is a topological group, it follows that $\Delta_{T/H}$ is closed.
\item (c): Let $I^\prime=I\times\mathfrak{N}_t(T)$, where $I^\prime$ is directed by $(i_1,U_1)\le(i_2,U_2)$ iff $i_1\le i_2$ and $U_1\supseteq U_2$ in $I^\prime$. For $i^\prime=(i,U)\in I^\prime$, by (a) we can choose $\xi(i^\prime)>i$ in $I$ such that $t_{\xi(i^\prime)}H\in\rho[U]$. Let $k_{i^\prime}\in U$ with $\rho k_{i^\prime}=\rho t_{\xi(i^\prime)}$. The proof is completed.
\end{proof}

Now if $T$ is LC and $H<T$ is co-compact, then for any compact $U\in\mathfrak{N}_e(T)$ we have that $\rho[U]\in\mathfrak{N}_{\rho(e)}(T/H)$ and there is a finite set $F\subset T$ with $T/H=\{suH\,|\,s\in F, u\in U\}$. Let $K=FU$. Then $K\in2^T$, $KH=T$ and $H$ is syndetic in $T$.
However, if $[T:H]$ is not finite and $T$ is not LC, the converse need not be true and so $\mathcal {S}^\textsl{ns}(T)\subsetneq\mathcal {S}^\textsl{ncc}(T)$ in general. For example, let $T$ be connected non-LC and $\chi\in\hat{T}$ (a continuous character of $T$) with $\chi\not\equiv 1$. Then $\ker\chi$ is a normal closed subgroup of $T$ such that $T/\ker\chi\cong\mathbb{S}$, where $\mathbb{S}$ is the unit circle in $\mathbb{C}$. So $T/\ker\chi$ is compact ($\ker\chi$ is co-compact in $T$) but we cannot guarantee $\ker\chi$ syndetic in $T$. If $\mathscr{X}$ is a minimal flow and $H=\ker\chi$, then it is of interest to know whether or not $H\curvearrowright X$ is a pointwise a.p. flow (see Thm.~\ref{3.1}).

\begin{se}\label{1.9}
A subgroup $H$ is called {\it subnormal} in $T$ (\cite[Def.~6.16]{E69}), if there exists a sequence of subgroups
$H_0=H<H_1<\dotsm<H_n=T$
such that $H_i$ is normal in $H_{i+1}$, denoted $H_i\lhd H_{i+1}$, for $0\le i<n$. By $\mathcal{S}^\textsl{sncc}(T)$ we denote the family of the subnormal co-compact closed subgroups of $T$.
\end{se}

\begin{se}\label{1.10}
As usual, $P(\mathscr{X})$ stands for the proximal relation of $\mathscr{X}$; that is, for every pair $(x,y)\in X\times X$, $(x,y)\in{P}(\mathscr{X})$ iff $\overline{T(x,y)}\cap \Delta_X\not=\emptyset$. If ${P}(\mathscr{X})=X\times X$, then $\mathscr{X}$ is called a \textit{proximal} flow. If $P(\mathscr{X})=\Delta_X$, then $\mathscr{X}$ is called a \textit{distal flow}. In this case, $\mathscr{X}$ is pointwise a.p.\,\cite{E69, G76, B79, A, De}.
\end{se}

\begin{exa}[{cf.~\cite[Ex.~6.19.2]{E69}}]\label{1.11}
Write $[x,y,z]$ for $\left[\begin{smallmatrix}1&x&y\\0&1&z\\0&0&1\end{smallmatrix}\right]$ and let
\begin{enumerate}
\item[] $G=\{[x,y,z]\,|\,x,y,z\in\mathbb{R}\}$,
\item[] $K=\{[x,y,z]\,|\,x,z\in\mathbb{Z}, y\in \mathbb{R}\}$,
\item[] $H=\{[x,y,z]\,|\,x,y,z\in\mathbb{Z}\}$.
\end{enumerate}
Then, under matrix multiplication $\cdot$, $H$ is a closed syndetic non-normal discrete subgroup of the simply connected nilpotent Lie group $G$ such that $H\lhd K\lhd G$. So $H\in\mathcal{S}^\textsl{sns}(G)$; but $H\notin\mathcal{S}^\textsl{ns}(G)$ for
\begin{enumerate}
\item[] $H\cdot[0,0,z]=\{[i,iz+j,z+k]\,|\,i,j,k\in\mathbb{Z}\}$ and
\item[] $[0,0,z]\cdot H=\{[i,j,z+k]\,|\,i,j,k\in\mathbb{Z}\}$
\end{enumerate}
such that $H\cdot[0,0,z]\not=[0,0,z]\cdot H$ for all $z\in\mathbb{R}\setminus\mathbb{Q}$.
The coset flow $G\curvearrowright G/H$ is of course minimal. Moreover, it is known that $G\curvearrowright G/H$ is a distal non-equicontinuous flow (see Thm.~\ref{3.4}B for a self-contained proof).
\end{exa}

\begin{lem}\label{1.12}
Let $\mathscr{X}$ be any flow and $H$ a co-compact closed subgroup of $T$, where $X$ need not be compact. If $X_0$ is an $H$-invariant closed subset of $X$, then $TX_0$ is a $T$-invariant closed subset of $X$ such that $t_iX_0\to tX_0$ whenever $t_iH\to tH$ in $T/H$.
\end{lem}

\begin{proof}
(If $H$ is syndetic, then $T=KH$ for some $K\in2^T$, and so that $TX_0=KHX_0=KX_0$ is closed, for $K$ is compact.)
\item \textit{Step~1.} Assume  $t_i\in T$ is any net with $t_iH\to tH$ in $T/H$ for some $t\in T$. By going to a subnet (Lem.~\ref{1.8}c) if necessary, we can select a net $s_i$ in $H$ such that $t_is_i\to t$ in $T$ so $(t_is_i)^{-1}\to t^{-1}$. Let $t_i\in T$ and $y_i\in X_0$ such that $t_iy_i\to x\in X$. To show $t_iX_0\to tX_0$, we need only prove $x\in tX_0$. As $(t_is_i)s_i^{-1}y_i\to x$ and
$(t_is_i)^{-1}(t_is_i)(s_i^{-1}y_i)=s_i^{-1}y_i\to t^{-1}x$, it follows that $t^{-1}x\in X_0$ and $x\in tX_0$.

\item \textit{Step~2.} Assume $H$ is co-compact. Let $t_i\in T$ and $y_i\in X_0$ such that $t_iy_i\to x\in X$. Since $T/H$ is compact, we can assume $t_iH\to tH$ for some $t\in T$. Then by Step~1, $x\in tX_0$ and thus, $TX_0$ is closed.
The proof is completed.
\end{proof}

\begin{rem}\label{1.13}
Condition ``co-compact'' is crucial for Lemma~\ref{1.12}. Otherwise, we can easily point out counterexamples; for instance, if $H=\{e\}$ and $x\in X$, then $Hx=\{x\}$ is $H$-invariant closed but $Tx$ is generally not closed whenever $T$ is not compact.
\end{rem}

\section{\bf Motivations}\label{s2}
Now it is time to formulate our motivations.

\begin{thm}[{inheritance by \textsl{ns}-subgroup; \cite[Thm.~4.04, 4.19]{GH}, \cite{E69, B79, A}}]\label{2.1}
Let $\mathscr{X}$ be a flow and $H\in\mathcal {S}^\textsl{ns}(T)$, then:
\begin{enumerate}[(1)]
\item[$(1)$] $\mathscr{X}$ is pointwise a.p. iff $H\curvearrowright X$ is a pointwise a.p. subflow of $\mathscr{X}$.
\item[$(2)$] If $\mathscr{X}$ is minimal, then $\mathscr{O}_H\colon X\rightarrow2^X$ is continuous.
\end{enumerate}
\end{thm}

In fact, if $T$ is only a para-topological group---namely, $T$ is Hausdorff and $(s,t)\mapsto st$ is continuous but $t\mapsto t^{-1}$ need not be continuous, this theorem is still true (cf.~\cite[Lem.~1.5e and Thm.~2.7]{D22}). For ``$H\in\mathcal {S}^\textsl{ncc}(T)$'' instead of condition ``$H\in\mathcal {S}^\textsl{ns}(T)$'' there is a useful partial generalization of Theorem~\ref{2.1} under an additional condition---proximal---as follows:

\begin{thm}[{\cite[Lem.~II.3.2]{G76}}]\label{2.2}
Let $\mathscr{X}$ be a minimal `proximal' flow. If $H\in\mathcal {S}^\textsl{ncc}(T)$, then $H\curvearrowright X$ is minimal proximal as a subflow of $\mathscr{X}$.
\end{thm}

See Theorem~\ref{4.2} and Corollary~\ref{4.4} below for generalizations of Theorem~\ref{2.2}.
Under another kind of additional condition---disjointness and co-compactness---in place of ``\textsl{ns}'', there is a result as follows:

\begin{thm}[{cf.~\cite[Prop.~IV.5.1]{G76}}]\label{2.3}
Let $H$ be a co-compact closed subgroup of $T$. Then
\begin{enumerate}
\item[] $T\curvearrowright X\times T/H$, defined by $(t, (x,sH))\mapsto(tx, tsH)$,
\end{enumerate}
is minimal iff $\mathscr{X}$ is minimal and $H\curvearrowright X$ is a minimal subflow of $\mathscr{X}$.
\end{thm}

\begin{proof}
\item Necessity: Clearly, $\mathscr{X}$ is minimal. For proving that $H\curvearrowright X$ is a minimal subflow of $\mathscr{X}$, let $X_0$ be a minimal subset of $H\curvearrowright X$ and simply write $[e]=H\in T/H$. Then as $H(X_0\times\{[e]\})=X_0\times\{[e]\}$, it follows by Lemma~\ref{1.12} that $T(X_0\times\{[e]\})$ is a closed invariant subset of $X\times T/H$ by $T$. So
$$T(X_0\times\{[e]\})=X\times T/H\supset X\times\{[e]\}.$$
As $t(x,[e])\in X\times\{[e]\}\Rightarrow t\in H$ and by $X\times\{[e]\}=\bigcup_{t\in H}tX_0\times\{[e]\}$, it follows easily that $X=HX_0=X_0$. Hence $H\curvearrowright X$ is a minimal subflow of $\mathscr{X}$.

\item Sufficiency: Let $W$ be a minimal subset of $X\times T/H$ under $T$. We can define a closed subset of $X$, $Y=\{x\in X\,|\, (x,[e])\in W\}\not=\emptyset$. By $HY=Y$ and minimality of $H\curvearrowright X$, we have that $Y=X$. Thus, $X\times\{[e]\}\subseteq W$, and then for any $t\in T$, $t(X\times\{[e]\})=X\times\{tH\}$. So $W=X\times T/H$. The proof is completed.
\end{proof}

The normality of $H$ plays only a role in the ``only if'' part of Theorem~\ref{2.1}-(1) (see Thm.~\ref{3.1}-(2) below). If $H$ is not normal, then this part is generally false by constructing a counterexample as follows:

\begin{exa}\label{2.4}
Let $T=\textrm{GL}(2,\mathbb{R})$, the general linear group with the usual topology, which is LC. Let
\begin{enumerate}
\item[] $H=\left\{\left[\begin{smallmatrix}x&0\\y&z\end{smallmatrix}\right]\,|\, x,y,z\in\mathbb{R}\textrm{ s.t. }xz\not=0\right\}$.
\end{enumerate}
Then $H$ is a non-normal closed subgroup of $T$. Moreover, $H$ is syndetic and so co-compact in $T$. Indeed, if
\begin{enumerate}
\item[] $K=\left\{\left[\begin{smallmatrix}a&b\\c&d\end{smallmatrix}\right]\,|\, a,b,c,d\in[-1,1]\textrm{ s.t. }ad-bc=1\right\}$,
\end{enumerate}
then $K\in2^T$ and $T=KH$ (see \cite[pp.392--393]{K67}). Let $T\curvearrowright\mathbb{R}^2$ be the noncompact flow with the phase mapping
\begin{enumerate}
\item[]
$T\times\mathbb{R}^2\rightarrow\mathbb{R}^2,\quad \left(\left[\begin{smallmatrix}a&b\\c&d\end{smallmatrix}\right],\left[\begin{smallmatrix}x\\y\end{smallmatrix}\right]\right)\mapsto
\left[\begin{smallmatrix}a&b\\c&d\end{smallmatrix}\right]\left[\begin{smallmatrix}x\\y\end{smallmatrix}\right]=\left[\begin{smallmatrix}ax+by\\cx+dy\end{smallmatrix}\right]$.
\end{enumerate}
Let $\mathbb{P}^1$ denote the real projective line space; i.e., the space of all $1$-dimensional linear subspaces of $\mathbb{R}^2$. Then the above natural action of $T$ on $\mathbb{R}^2$ induces an action of $T$ on $\mathbb{P}^1$. It is clear that
\begin{enumerate}[(1)]
\item $T\curvearrowright\mathbb{P}^1$, with phase mapping $T\times\mathbb{P}^1\rightarrow\mathbb{P}^1$, is a homogeneous flow (i.e., $Tx=\mathbb{P}^1\ \forall x\in \mathbb{P}^1$).
\end{enumerate}
Let $x\in\mathbb{P}^1$. Since
$\left[\begin{smallmatrix}1&k\\0&1\end{smallmatrix}\right]x\rightarrow L$ as $k\to\infty$,
where $L\in\mathbb{P}^1$ is the line through the points $(-1,0)$ and $(1,0)$ of $\mathbb{R}^2$. Thus:
\begin{enumerate}[(2)]
\item $T\curvearrowright\mathbb{P}^1$ is a minimal proximal flow.
\end{enumerate}
We now consider the induced subflow $H\curvearrowright\mathbb{P}^1\colon H\times\mathbb{P}^1\rightarrow\mathbb{P}^1$. We can assert the following:
\begin{enumerate}[(3)]
\item $H\curvearrowright\mathbb{P}^1$ is proximal but not pointwise a.p. (and hence not minimal).
\end{enumerate}
\begin{proof}
Suppose to the contrary that $H\curvearrowright\mathbb{P}^1$ is pointwise a.p. Since $H$ is a syndetic subgroup of $T$, $H\curvearrowright\mathbb{P}^1$ is proximal by (2) so $\mathbb{P}^1$ contains a unique minimal set under $H$. Thus, $H\curvearrowright\mathbb{P}^1$ is minimal. However, since the line through the points $(0,1)$ and $(0,-1)$ of $\mathbb{R}^2$ is a fixed point under $H$, $\mathbb{P}^1=\{pt\}$, contrary to $\dim\mathbb{P}^1=1$. The proof is completed.
\end{proof}

\begin{enumerate}[(4)]
\item $T\curvearrowright\mathbb{P}^1\times T/H$ is not a minimal flow.
\end{enumerate}
\begin{proof}
By Theorem~\ref{2.3} and (3) above.
\end{proof}
In fact, \textit{$H$ is not subnormal in $T$} according to (2), (3) and Theorem~\ref{3.1} below.
\end{exa}

Note that if $\mathscr{X}$ is not equicontinuous/a.p., then the minimality condition of $\mathscr{X}$ is critical for Theorem~\ref{2.1}-(2), as shown by the following

\begin{thm}\label{2.5}
There exists a flow $\mathbb{R}\curvearrowright X\colon \mathbb{R}\times X\xrightarrow{(t,x)\mapsto t\cdot x}X$ such that
$\mathscr{O}_{\mathbb{R}}\colon x\mapsto \overline{\mathbb{R}\cdot x}$ is continuous (so it is pointwise a.p.; see Thm.~\ref{3.5}A); however,
$\mathscr{O}_\mathbb{Z}\colon x\mapsto\overline{\mathbb{Z}\cdot x}$ is not continuous from $X$ into $2^X$.
(Note here that $\mathbb{Z}\in\mathcal{S}^\textsl{ns}(\mathbb{R})$.)
\end{thm}

\begin{proof}
Let
$X=\{z\in\mathbb{C}\colon |z|\le 1\}$. Then we can write $X={\bigcup}_{0\le r\le 1}\mathbb{S}_r$,
where $\mathbb{S}_r=\{z\in\mathbb{C}\colon|z|=r\}$ for all $r\in[0,1]$. The continuous-time flow $\mathbb{R}\times X\rightarrow X$ is defined by
$(t,x)\mapsto t\cdot x=e^{2\pi i|x|t}x$.
Clearly, $\overline{\mathbb{R}\cdot x}=\mathbb{S}_{|x|}$ for all $x\in X$. If $x_i\to x$ in $X$, then  $\overline{\mathbb{R}\cdot x_i}\to\overline{\mathbb{R}\cdot x}$. However, if $r$ with $0<r<1$ is rational and $\{r_i\}$ is a sequence of irrational numbers with $r_i\to r$, then  for $x_i\in\mathbb{S}_{r_i}\to x\in\mathbb{S}_r$, as $i\to\infty$,
$$
\overline{\mathbb{Z}\cdot x_i}=\mathbb{S}_{r_i}\to\mathbb{S}_r\not=\overline{\mathbb{Z}\cdot x}.
$$
Thus, $\mathscr{O}_\mathbb{Z}\colon X\ni x\mapsto\overline{\mathbb{Z}\cdot x}\in 2^X$ is not a continuous mapping. The proof is completed.
\end{proof}

\section{\bf Inheritance theorems}\label{s3}
Example~\ref{2.4} also shows that the disjointness condition in Theorem~\ref{2.3} is crucial.
The main purpose of this paper is to generalize Theorem~\ref{2.1} without the proximality and disjointness as follows:

\begin{thm}[inheritance by \textsl{sncc}-subgroup]\label{3.1}
Let $\mathscr{X}$ be any flow and $H$ a closed subgroup of $T$, let $H\curvearrowright X$ be a subflow of $\mathscr{X}$. Then:
\begin{enumerate}[(1)]
\item[$(1)$] If $H\in\mathcal {S}^\textsl{sncc}(T)$ and $x\in X$ is a.p. for $T\curvearrowright X$, then $x$ is a.p. for $H\curvearrowright X$.
\item[$(2)$] If $H$ is co-compact in $T$ and $x\in X$ is a.p. for $H\curvearrowright X$, then $x$ is a.p. for $T\curvearrowright X$.
\item[$(3)$] If $\mathscr{X}$ is minimal and $H\in\mathcal {S}^\textsl{ncc}(T)$, then $\{\overline{Hx}\,|\,x\in X\}$ is a $T$-invariant partition of $X$ such that $\mathscr{O}_H\colon X\rightarrow2^X$ is continuous.
\end{enumerate}
Consequently, for every $H\in\mathcal {S}^\textsl{sncc}(T)$, a point $x$ in $X$ is a.p. for $T\curvearrowright X$ iff it is a.p. for $H\curvearrowright X$.
\end{thm}

\begin{proof}
\item (1): Suppose $x\in X$ is a.p. for $T$. First, assume $H\in\mathcal {S}^\textsl{ncc}(T)$. By Zorn's lemma, it follows that we can select a point $y\in \overline{Hx}$ that is a.p. for $H$. Since $\overline{Tx}$ is a minimal set for $T$ and $y\in \overline{Tx}$, there is a net $\{t_n\,|\,n\in D\}$ in $T$ such that $t_ny\to x$. Since $T/H$ is a compact group, we can find an element $k\in T$ such that $t_nH\to kH$ in $T/H$. By Lemma~\ref{1.8}c, we can assume that there is a net $k_n\in T$ such that $k_n\to k$ in $T$ and $t_nH=k_nH$ for all $n\in D$. By $t_nH=k_nH$, it follows that $k_n^{-1}t_n\in H$ for all $n\in D$. Thus $t_ny\to x$ and $k_n^{-1}\to k^{-1}$ implies that $k_n^{-1} t_ny\to k^{-1}x$ and $k^{-1}x\in \overline{Hy}$ is a.p. for $H$. Finally, by normality of $H$, $x$ is a.p. for $H$. Indeed, for each point $z\in\overline{Hx}$, we have that $k^{-1}z\in\overline{Hk^{-1}x}$ and $k^{-1}\overline{Hz}=\overline{Hk^{-1}x}$, and, $\overline{Hz}=\overline{Hx}$. Thus $\overline{Hx}$ is minimal for $H$. Since $\overline{Hx}$ is compact Hausdorff, $x$ is a.p. for $H$.

Next, suppose $H\in\mathcal {S}^\textsl{sncc}(T)$. Then there exists a finite sequence of closed subgroups of $T$:
$H_0=H\lhd H_1\lhd\dotsm\lhd H_n=T$, as in Def.~\ref{1.9}. Clearly, we have that $H_i\in \mathcal{S}^\textsl{ncc}(H_{i+1})$. So by induction, it follows easily that $x$ is an a.p. point for $H$.

\item (2): Suppose $x\in X$ is an a.p. point for $H$. Without loss of generality we may assume $X=\overline{Tx}$ and by Zorn's lemma we can select an a.p. point $y\in X$ under $T$. Then every point of $\overline{Hy}$ is a.p. under $T$ for $\overline{Hy}\subseteq\overline{Ty}$. Take a net $t_i\in T$ with $t_ix\to y$.
Since $T/H$ is a compact coset space, by Lemma~\ref{1.8}c we can take elements $s_i,k\in T$ such that (a subnet of)
$s_i\to k$ (so $s_i^{-1}\to k^{-1}$) in $T$ and $s_iH=t_iH$. Then $s_i^{-1}t_ix\to z:=k^{-1}y$ and $s_i^{-1}t_i\in H$. Thus, $z\in \overline{Hx}$ and $z$ is a.p. for $T$. Then $\overline{Hx}=\overline{Hz}\,(\subseteq\overline{Tz})$ consists of $T$-a.p. points. Whence $x$ is an a.p. point for $T$.

\item (3): By (1) and normality of $H$, $\{\overline{Hx}\,|\,x\in X\}$ is a $T$-invariant partition of $X$ into $H$-minimal sets. Now let $x_n\to x$ in $X$. We need only prove that $\overline{Hx_n}\to \overline{Hx}$ in $2^X$. Suppose this is not true; then we may assume (a subnet of) $\overline{Hx_n}\to L\in 2^X$ such that  $\overline{Hx}\subsetneq L$. Since $\mathscr{O}_H\colon X\rightarrow2^X$ is lower semi-continuous and $X=\overline{\bigcup\{t\overline{Hx}\,|\,t\in T\}}$, hence for $\varepsilon_n\in\mathscr{U}$ with $\varepsilon_n\to\Delta_X$, we can choose $t_n\in T$, with $t_nx$ sufficiently close to $x_n$, such that $\varepsilon_n[t_n\overline{Hx}]\supseteq \overline{Hx_n}$ for all $n$. By compactness of $T/H$ and using Lemma~\ref{1.8}c, we can assume $t_n\to k$ in $T$ and $t_nH\to kH$ in $T/H$ so that (a subnet of) $t_n\overline{Hx}\to k\overline{Hx}$ in $2^X$. Then $\overline{Hx}\subsetneq L\subseteq k\overline{Hx}=\overline{Hkx}$, contrary to that $\{\overline{Hz}\,|\,z\in X\}$ is a partition of $X$ into $H$-minimal sets.
The proof is completed.
\end{proof}

\begin{cor}[{\cite[Thm.~4.04]{GH} and \cite[Prop.~2.8]{E69}}]\label{3.2}
Let $\mathscr{X}$ be a flow and $H\in\mathcal {S}^\textsl{ns}(T)$, then $x\in X$ is a.p. for $T\curvearrowright X$ iff $x$ is a.p. for $H\curvearrowright X$.
\end{cor}

Corollary~\ref{3.2} is still true if $X$ is only LC instead of $X$ being compact; see \cite[Prop.~2.8]{E69}. In fact, if $X$ is LC and $x$ is a.p. for $H\in\mathcal {S}^\textsl{ns}(T)$, then $\overline{Hx}$ is compact and $\overline{Tx}=K\overline{Hx}$ for some $K\in2^T$ so that $\overline{Tx}$ is a compact set in $X$. Although the compactness of $X$ plays a role in our proof of Theorem~\ref{3.1}-(2), it turns out that we can conclude the following in our \textsl{sncc}-setting.

\begin{prop}\label{3.3}
Let $\mathscr{X}$ be a flow, with $X$ an LC Hausdorff space not necessarily compact. Then:
\begin{enumerate}[(1)]
\item[$(1)$] If $H$ is a co-compact closed subgroup of $T$ and $x\in X$ is a.p. for $H\curvearrowright X$, then $\overline{Tx}$ is compact minimal under $T$.
\item[$(2)$] If $H\in\mathcal {S}^\textsl{sncc}(T)$, then $x\in X$ is a.p. for $T\curvearrowright X$ iff $x$ is a.p. for $H\curvearrowright X$.
\end{enumerate}
\end{prop}

\begin{proof}
\item (1): Let $x\in X$ be a.p. for $H\curvearrowright X$. Since $X$ is LC, hence by a standard argument it follows that $\overline{Hx}$ is a compact set in $X$.
\item \textit{Step~1.} If $\{t_nx\}$ is a net in $Tx$, then there exists a subnet $\{s_i\}$ of $\{t_n\}$ such that $s_ix\to y$ for some $y\in X$.
Indeed, we can select a subnet $\{s_i\}$ of $\{t_n\}$ and some $k\in T$ such that $s_iH\to kH$ in $T/H$, for $T/H$ is compact. Further by Lemma~\ref{1.8}c, we can select a net $\{k_i\}$ in $T$ such that $k_i\to k$ and $k_iH=s_iH$. Then
$Hx\ni k_i^{-1}s_ix\to y^\prime\in \overline{Hx}$ and $s_ix\to y:=ky^\prime\in X$.

\item \textit{Step~2.} To prove that $\overline{Tx}$ is compact, it suffices to show that there exists a convergent subnet from every net $\{y_i\colon i\in I\}$ in $\overline{Tx}$. Indeed, given $\{y_i\colon i\in I\}$ in $\overline{Tx}$, let $\mathcal {U}$ be a uniformity on $X$ (any will do and there are usually many). Let
    $I^\prime=I\times\mathcal {U}$. For $i^\prime=(i,\alpha)\in I^\prime$ let $t_{i^\prime}\in I$ such that $(y_i,t_{i^\prime}x)\in\alpha$. By Step~1, there is a subnet $\{t_{i^{\prime\prime}}\}$ of $\{t_{i^\prime}\}$ such that $t_{i^{\prime\prime}}x\to y$ for some $y\in\overline{Tx}$. For any $\varepsilon\in\mathcal {U}$ choose $\alpha\in\mathcal {U}$ with $\alpha\circ\alpha\subseteq\varepsilon$. Eventually $t_{i^{\prime\prime}}x\in\alpha[y]$ and eventually $(y_{i^{\prime\prime}}, t_{i^{\prime\prime}}x)\in\alpha$. Thus, $y_{i^{\prime\prime}}\to y$ and $\overline{Tx}$ is compact.
\item \textit{Step~3.} By Theorem~\ref{3.1}, it follows that $x$ is a.p. and so $\overline{Tx}$ is compact minimal for $T$. This proves (1).

\item (2): Obvious by (1) and Theorem~\ref{3.1}. The proof is completed.
\end{proof}

\begin{que}\label{3.4}
{\it If $\mathscr{X}$ is a minimal flow and $H\in\mathcal{S}^\textsl{sncc}(T)$ not normal in $T$, is $\mathscr{O}_H\colon X\rightarrow2^X$ a continuous map?} The general answer is NO!
\end{que}

\begin{3.4A}
Let $G$ and $H$ be as in Example~\ref{1.11}. So $H\in\mathcal {S}^\textsl{sncc}(G)$ and $H\notin\mathcal {S}^\textsl{ncc}(G)$. Let $X=G/H$. Recall that $G\curvearrowright X$ is a minimal flow. Then by Theorem~\ref{3.1}, $H\curvearrowright X$ is pointwise a.p. Given $g\in G$, write $[g]=gH\in X$. Clearly, if $g\in K$, then $\overline{H[g]}=\{[g]\}$ because $Hg=gH$. As for all $z\in\mathbb{R}$ and $i,j,k, i^\prime,j^\prime,k^\prime\in\mathbb{Z}$
$$
[i,j,k]\cdot[0,0,z]\cdot[i^\prime,j^\prime,k^\prime]=\begin{pmatrix}1&i^\prime+i&j^\prime+k^\prime i+iz+j\\0&1&k^\prime+z+k\\0&0&1\end{pmatrix},
$$
it follows that if we choose a sequence $g_n=[0,0,1/2n]\in G$ with $n\in\mathbb{N}$, then $g_n\to[0,0,0]$ the unit matrix and $\overline{H[g_n]}\not\to\overline{H[e]}=\{[e]\}$ in $2^X$ as $n\to\infty$. Thus, $\mathscr{O}_H\colon X\rightarrow2^X$ is not a continuous map and then by Theorem~\ref{3.1}, $H\curvearrowright X$ is pointwise a.p. non-minimal.
\end{3.4A}

\begin{3.4B}[{due to L.~Auslander, F.~Hahn and L.~Markus 1963}]
Let $G$, $H$ be as in Example~\ref{1.11}. Then $G\curvearrowright G/H$ is a non-equicontinuous distal coset flow.
\end{3.4B}

\begin{proof}
First $G\curvearrowright G/H$ is not equicontinuous; otherwise, $\mathscr{O}_H\colon X\rightarrow2^X$ would be continuous, contrary to Example~\ref{3.4}A. As
$H\lhd K\lhd G$, it follows that $G/H\xrightarrow{\pi_1}G/K\xrightarrow{\pi_2}\{pt\}$ are group extensions.\footnote{In fact, $\pi_2$ is a group extension of $\{pt\}$ for $K\lhd G$ and $G/K$ is a compact Hausdorff topological group. On the other hand, $K/H$ is a compact Hausdorff topological group such that $G\curvearrowright G/H\curvearrowleft K/H$, given by $(t,gH, kH)\mapsto tgkH$ for all $t\in G$, $gH\in G/H$ and $kH\in K/H$, is a bi-flow. Now for $gH\in G/H$, $$\pi_1^{-1}(gH)=gKH=\textrm{Aut}_{\pi_1}(G\curvearrowright G/H)(gH),\quad \textrm{Aut}_{\pi_1}(G\curvearrowright G/H)\cong K/H\cong\mathbb{R}/\mathbb{Z}.$$
Thus, $\pi_1$ is a group extension.} So $G\curvearrowright G/H$ is a distal coset flow. \end{proof}

\begin{se}[Weakly a.p. in the sense of W.\,H. Gottschalk]\label{3.5}
	A flow $\mathscr{X}$ is called {\it weakly a.p.} \cite{G46}, if given $\alpha\in\mathscr{U}$ there exists a compact set $K$ in $T$ such that $Tx\subseteq K^{-1}\alpha[x]$ for all $x\in X$.
	
\begin{3.5A}[{cf.~\cite[Thm.~5]{G46} or \cite[Thm.~4.24]{GH}}]
A flow $\mathscr{X}$ is weakly a.p. iff $\mathscr{O}_T\colon X\rightarrow 2^X$ is continuous.
\end{3.5A}

\begin{3.5B}
A minimal flow is weakly almost periodic.
\end{3.5B}

\begin{3.5C}
Let $\mathscr{X}$ be a minimal flow and $H\in\mathcal {S}^\textsl{ncc}(T)$. Then $H\curvearrowright X$, as a subflow of $\mathscr{X}$, is weakly a.p. of Gottschalk.
\end{3.5C}

\begin{proof}
By Theorem~\ref{3.1}-(3) and Theorem~\ref{3.5}A.
\end{proof}
\end{se}

Thus, to prove that $\mathscr{O}_H\colon X\rightarrow2^X$ continuous in Question~\ref{3.4}, it is enough to show that $H\curvearrowright X$ is weakly a.p. in the sense of Gottschalk.

\begin{cor}\label{3.6}
	Let $\mathscr{X}$ be a minimal flow and $H\in\mathcal{S}^\textsl{sncc}(T)$ with a subnormal sequence
	$H_0=H\lhd H_1\lhd\dotsm\lhd H_n=T$ such that $[H_{i+1}: H_i]$ is finite for $1\le i<n$. Then $\mathscr{O}_H\colon X\rightarrow2^X$ is continuous.
\end{cor}

\begin{proof}
By induction and Theorem~\ref{3.1}, it follows that $\{\overline{H_1x}\,|\,x\in X\}$ is an $H$-invariant finite partition of $X$. By Theorem~\ref{3.1} again, it follows that $\mathscr{O}_H\colon \overline{H_1x}\rightarrow2^{\overline{H_1x}}\subseteq2^X$ is continuous. Thus, $\mathscr{O}_H\colon X\rightarrow2^X$ is continuous for every $\overline{H_1x}$ is clopen in $X$. The proof is completed.
\end{proof}

If $\mathscr{X}$ is a minimal flow and $H\in\mathcal {S}^\textsl{ncc}(T)$, then Theorem~\ref{3.1} tells us that $\mathcal {P}=\{\overline{Hx}\,|\,x\in X\}$ is a $T$-invariant closed partition of $X$. In fact, we shall prove that $T\curvearrowright X/\mathcal {P}$ is an a.p. factor of $\mathscr{X}$\,(see Thm.~\ref{6.5}).

\section{\bf Homomorphism of flows}\label{s4}

\begin{se}\label{4.1}
Let $\phi\colon\mathscr{X}\rightarrow\mathscr{Z}$ be a homomorphism of flows; i.e., $\phi\colon X\rightarrow Z$ is a continuous mapping such that $\phi tx=t\phi x$ for all $x\in X$ and all $t\in T$. If in addition $\phi X=Z$, then $\mathscr{X}$ is called an extension of $\mathscr{Z}$ and $\mathscr{Z}$ is called a factor of $\mathscr{X}$.
\item[\;\;\textbf{\ref{4.1}a.}] Write ${R}_\phi=\{(x_1,x_2)\,|\,x_1,x_2\in X, \phi x_1=\phi x_2\}$.

\item[\;\;\textbf{\ref{4.1}b.}] We say that $(x_1,x_2)\in{P}_{\!\phi}(\mathscr{X})$ iff $(x_1,x_2)\in{R}_\phi\cap{P}(\mathscr{X})$. Then ${P}(\mathscr{X})={P}_{\!\phi}(\mathscr{X})$ if $Z=\{pt\}$. Thus, $\mathscr{X}$ is {\it $\phi$-proximal}, or simply, {\it $\phi$ is proximal}, iff ${P}_{\!\phi}(\mathscr{X})={R}_\phi$.
\item[\;\;\textbf{\ref{4.1}c.}] We say that $\mathscr{X}$ is \textit{$\phi$-distal}, or simply $\phi$ is \textit{distal}, iff ${P}_{\!\phi}(\mathscr{X})=\Delta_X$ iff $\overline{T(x_1,x_2)}\cap\Delta_X=\emptyset$ for all $(x_1,x_2)\in{R}_\phi\setminus\Delta_X$.

\item[\;\;\textbf{\ref{4.1}d.}] Let $H\in\mathcal {S}^\textsl{sncc}(T)$ with $H=H_0\lhd H_1\lhd\dotsm\lhd H_{n-1}\lhd H_n=T$. Then
\begin{enumerate}
\item[] $T/H\rightarrow T/H_1\rightarrow\dotsm\rightarrow T/H_{n-1}\rightarrow T/H_n=\{pt\}$
\end{enumerate}
is a sequence of a.p. extension. Thus, $T\curvearrowright T/H$ is a distal coset flow, not necessarily a.p. (see Thm.~\ref{3.4}B).
\end{se}

Then the following result together with Theorem~\ref{3.1} may give us a relativization of Theorem~\ref{2.2}.

\begin{thm}[{inheritance of relativized proximality; cf.~\cite[Thm.~10.03]{GH} for $Z=\{pt\}$, $H$ syndetic}]\label{4.2}
Let $\phi\colon\mathscr{X}\rightarrow\mathscr{Z}$ be a homomorphism of flows and $H$ a co-compact closed subgroup of $T$.
Then $(x_1,x_2)\in{P}_{\!\phi}(T\curvearrowright X)$ iff $(x_1,x_2)\in{P}_{\!\phi}(H\curvearrowright X)$.
\end{thm}

\begin{proof}
Let $(x_1,x_2)\in{P}_{\!\phi}$ under $T$, then there is a net $\{t_n\,|\,n\in D\}$ in $T$ and a point $x\in X$ such that $t_n(x_1,x_2)\to(x,x)$. Since $T/H$ is compact, we can assume (a subnet of) $t_nH\to kH$ for some element $k\in T$. Further by Lemma~\ref{1.8}c, we can assume that there is a net $k_n\in T$ with $k_n\to k$ in $T$ such that $t_nH=k_nH$ for all $n\in D$. Then $k_n^{-1}t_n\in H$ and $k_n^{-1}\to k^{-1}$. Thus $k_n^{-1}t_n(x_1,x_2)\to k^{-1}(x,x)\in\Delta_X$ and $(x_1,x_2)\in{P}_{\!\phi}$ under $H$.
On the other hand, obviously, ${P}_{\!\phi}(H\curvearrowright X)\subseteq{P}_{\!\phi}(T\curvearrowright X)$.
The proof is completed.
\end{proof}

\begin{thm}[{inheritance of relativized distality; cf.~\cite[Prop.~5.14]{E69} for $H$ syndetic, $Z=\{pt\}$}]\label{4.3}
Let $\phi\colon\mathscr{X}\rightarrow\mathscr{Z}$ be a homomorphism of flows and $H$ a co-compact closed subgroup of $T$. Then $\mathscr{X}$ is $\phi$-distal iff $H\curvearrowright X$ is $\phi$-distal.
\end{thm}

\begin{proof}
By Theorem~\ref{4.2}.
\end{proof}

\begin{cor}\label{4.4}
Let $\mathscr{X}$ be a minimal proximal flow. If $H\in\mathcal {S}^\textsl{sncc}(T)$, then $H\curvearrowright X$ is minimal proximal as a subflow of $\mathscr{X}$.
\end{cor}

\begin{proof}
By Theorem~\ref{3.1}, $H\curvearrowright X$ is pointwise a.p.; and by Theorem~\ref{4.2}, $H\curvearrowright X$ is proximal. Thus, $H\curvearrowright X$ is minimal proximal.
\end{proof}

\section{\bf Point-envelopes of subgroups}\label{s5}
\setcounter{thm}{-1}
\begin{se}\label{5.0}
Let $\mathscr{X}$ be a flow, $x\in X$, and $H$ a subset of $T$, as in \cite[Def.~2.08]{GH}, the \textit{$x$-envelope of $H$} in $T$ is defined by
$E_x[H]=\{t\in T\,|\,tx\in\overline{Hx}\}$.
\end{se}

Clearly, $H\subseteq E_x[H]$ and then $\overline{Hx}\subseteq\overline{E_x[H]x}\subseteq\overline{Hx}$ from which it follows that
\begin{enumerate}[\ref{5.0}\textbf{a.}]
\item $E_x[H]=E_x[E_x[H]]\ \forall x\in X$.
\end{enumerate}
If $H\lhd T$, then $tx\in\overline{Hx}$ implies $tHx=Htx\subseteq\overline{Hx}$; thus,
\begin{enumerate}[\ref{5.0}\textbf{b.}]
\item $t\in E_x[H]$ iff $t\overline{Hx}\subseteq\overline{Hx}$.
\end{enumerate}
Next we will be concerned with the case $H\in \mathcal{S}^{\textsl{ncc}}(T)$. Clearly, $E_x[H]$ is a closed subsemigroup of $T$ for this case. However, {\it is it a subgroup of $T$?} First, using Theorem~\ref{3.1} we can easily obtain the following result.

\begin{prop}\label{5.1}
If $\mathscr{X}$ is a minimal flow with $T$ abelian and $H<T$ is co-compact,  then $E_x[H]=E_y[H]$ for all $x,y\in X$.
\end{prop}

\begin{proof}
Let $x,y\in X$ and $H\in\mathcal {S}^{\textsl{ncc}}(T)$. Let $t\in E_x[H]$ and let $s_n\in T$ such that $s_nx\to y$. Then by Theorem~\ref{3.1}-(3),
$$
ty=\lim s_ntx\in\lim s_n\overline{Hx}=\lim\overline{Hs_nx}=\overline{Hy}.
$$
So $t\in E_y[H]$. Symmetrically, $E_y[H]\subseteq E_x[H]$. The proof is complete.
\end{proof}

\begin{prop}\label{5.2}
Let $\mathscr{X}$ be a flow and $x\in X$ an a.p. point. Let $H\in\mathcal {S}^\textsl{ncc}(T)$. Then $E_x[H]=E_y[H]$ for all $y\in\overline{Hx}$.
\end{prop}

\begin{proof}
Let $x\in X$ be a.p. for $T$. Then $x$ is a.p. for $H$ by Theorem~\ref{3.1} and so $\overline{Hx}=\overline{Hy}$ for all $y\in\overline{Hx}$.
It then follows from \ref{5.0}b that $E_x[H]=E_y[H]$ for all $y\in\overline{Hx}$. The proof is completed.
\end{proof}

Now we will prove that $E_x[H]$ is a closed subgroup of $T$. However, it is ``non-normal'' in general. The following theorem has the flavor of \cite[Cor.~2.12]{E69} that asserts that a closed subsemigroup that contains a normal syndetic subgroup is a subgroup of $T$.

\begin{thm}\label{5.3}
Let $T$ be an \textsl{w\!F}-group (cf.~Lem.~\ref{1.7}). If $S$ be a closed subsemigroup of $T$ such that for all $V\in\mathfrak{N}_e(T)$ there is a finite set $F$ in $T$ with $FVS=T$, then $S$ is a subgroup of $T$.
\end{thm}

\begin{note*}
If $T$ is a topological group, then `a finite set $F$' may be relaxed as `a compact set $F$' here.
\end{note*}

\begin{proof}
Let $s\in S$. We need only prove $s^{-1}\in S$. For this, take $V\in\mathfrak{N}_e(T)$. Then there is a finite subset $F$ of $T$ such that $FVS=T$. Let $k_0\in FV$; then $k_0s^{-1}=k_1s_1$ for some $k_1\in FV$ and $s_1\in S$. Further, $k_1s^{-1}=k_2s_2$ for some $k_2\in FV$ and $s_2\in S$. By induction, there are two sequences $k_0,k_1,k_2,\dotsc$ in $FV$ and $s_1,s_2,s_3,\dotsc$ in $S$ with $k_is^{-1}=k_{i+1}s_{i+1}$ for all $i=0,1,\dotsc$. Write $k_i=f_iv_i$ with $f_i\in F$ and $v_i\in V$ for all $i$. Since $F$ is finite, we can find some $0\le m<n<\infty$ such that $f_m=f_n=f$. Then
\begin{equation*}\begin{split}
k_n^{-1}k_ms^{-1}&=k_n^{-1}k_{m+1}s_{m+1}=k_n^{-1}(k_{m+1}s^{-1})ss_{m+1}\\
&=k_n^{-1}k_{m+2}s_{m+2}ss_{m+1}\\
&=k_n^{-1}(k_{m+2}s^{-1})ss_{m+2}ss_{m+1}=\dotsm\\
&=k_n^{-1}k_ns_nss_{n-1}\dotsm ss_{m+1}\\
&=s_nss_{n-1}\dotsm ss_{m+1}.
\end{split}
\end{equation*}
Hence $k_n^{-1}k_ms^{-1}\in S$ and $v_n^{-1}v_ms^{-1}\in S$.
Thus $v_ms^{-1}\in v_nS\subseteq VS$. Taking $v_m\to e$ in $T$, this implies that $s^{-1}\in VS$  for all $V\in\mathfrak{N}_e(T)$. Next we can select two nets $v_i\to e$ in $T$ and $s_i\in S$ such that $s^{-1}=v_is_i$ and then $v_i^{-1}s^{-1}=s_i\to s^{-1}\in S$, for $S$ is closed. The proof is completed.
\end{proof}

\begin{5.3A}
Let $S$ be a closed subsemigroup of an \textsl{w\!F}-group $T$, which includes a co-compact closed subgroup of $T$. Then $S$ is a co-compact closed subgroup of $T$.
\end{5.3A}

\begin{proof}
Let $H\le S$ be a co-compact closed subgroup of $T$. Let $V\in\mathfrak{N}_e(T)$. Then there is a finite set $F$ in $T$ with $FVH\,(\subseteq FVS)=T$. Thus, $S$ is a subgroup of $T$ by Theorem~\ref{5.3}. Moreover, by Lemma~\ref{1.7}, $S$ is co-compact.
\end{proof}

\begin{5.3B}[{cf.~\cite[Lem.~2.06]{GH} or \cite[Lem.~2..8.17]{B79} for $T$ a Hausdorff topological group}]
If $S$ is a syndetic closed subsemigroup of an \textsl{w\!F}-group $T$, then $S$ is a syndetic closed subgroup of $T$.
\end{5.3B}

\begin{proof}
Let $V\in\mathfrak{N}_e(T)$. Since $S$ is syndetic, there is a set $K\in2^T$ with $KS=T$. Further, we can find a finite set $F\subseteq T$ with $FV\supseteq K$. Thus, $FVS=T$. Then $S$ is a syndetic closed subgroup of $T$ by Theorem~\ref{5.3}.
\end{proof}

\begin{cor}[{comparable with \cite[Cor.~2.11]{E69}}]\label{5.4}
Let $E$ be an \textsl{F}-group and $M$ a closed subsemigroup of $E$. Then $M$ is a subgroup of $E$.
\end{cor}

\begin{note*}
In \cite[Cor.~2.11]{E69}, $E$ is a compact $T_1$ left(or right)-topological group. Then $M$ is a compact $T_1$ left(or right)-topological semigroup, and so that $M$ contains the unique idempotent $e$ (\cite[Lem.~2.9]{E69}). In view of this, Ellis' argument is invalid for our setting.
\end{note*}

\begin{proof}
As $E$ is compact, it follows that for every $V\in\mathfrak{N}_e(E)$, there is a finite set $F\subseteq E$ with $FV=E$, and so that $FVM=EM=E$. Thus, $M$ is a subgroup of $E$ by Theorem~\ref{5.3}.
\end{proof}

\begin{cor}[{cf.~\cite[Lem.~2.10-(1)]{GH} or \cite[Lem.~2.8.19]{B79}}]\label{5.5}
Let $\mathscr{X}$ be a flow, $x\in X$, and $A\in\mathcal{S}^\textsl{ns}(T)$. Then $E_x[A]$ is closed syndetic subgroup of $T$ with $A\subseteq E_x[A]$ and $\overline{Ax}=\overline{E_x[A]x}$.
\end{cor}

\begin{cor}[{cf.~\cite[Lem.~2.10-(3)]{GH} for $T$ an LC group}]\label{5.6}
Let $\mathscr{X}$ be any flow, $x\in X$ and $S\in\mathcal{S}^\textsl{ns}(T)$. If $U\in\mathfrak{N}_x(X)$, then there exists $K\in 2^T$ such that $Kx\subseteq U$ and  $E_x[S]=K^{-1}S$.
\end{cor}

\begin{proof}
We may assume $U\in\mathfrak{N}_x(X)$ is open. Let $L\in 2^T$ with $T=LS$. Set $H=L\cap E_x[S]$. Since $E_x[S]$ is a closed subgroup of $T$ and $S\subseteq E_x[S]$ by Corollary~\ref{5.5}, $H\not=\emptyset$ is a compact set with $E_x[S]=HS\,(=SH)$.

Let $t\in H$. Then $tx\in\overline{Sx}$ and $x\in\overline{t^{-1}Sx}$. So there exists an element $s_t\in S$ with $t^{-1}s_tx\in U$. Further, there exists a compact set $V_t\in\mathfrak{N}_t(H)$ such that $V_t^{-1}s_tx\subseteq U$.
Therefore we can find a finite set $\{t_1,\dotsc,t_m\}$ in $H$, which is such that $H=\bigcup_{i=1}^mV_{t_i}$ and $V_{t_i}^{-1}s_{t_i}x\subseteq U$ for $i=1,\dotsc,m$. Now we put $K=\bigcup_{i=1}^mV_{t_i}^{-1}s_{t_i}$. Then $K$ is compact in $T$ with $Kx\subseteq U$. Moreover,
$E_x[S]=S\left({\bigcup}_{i=1}^mV_{t_i}\right)=S\left({\bigcup}_{i=1}^ms_{t_i}^{-1}V_{t_i}\right)=SK^{-1}=K^{-1}S$.
The proof is completed.
\end{proof}

\begin{cor}\label{5.7}
Let $\mathscr{X}$ be any flow, $x\in X$, and $H\in\mathcal{S}^\textsl{ncc}(T)$. Then $E_x[H]$ is co-compact closed subgroup of $T$ and $T\curvearrowright T/E_x[H]$ is an a.p. minimal flow.
\end{cor}

\begin{proof}
Clearly $E_x[H]$ is a co-compact closed subgroup of $T$ such that $H\subseteq E_x[H]$ by Theorem~\ref{5.3}. Since $T/E_x[H]$ is a factor of $T/H$, hence $T\curvearrowright T/E_x[H]$ is an a.p. coset flow. The proof is completed.
\end{proof}

\section{\bf Relation of orbital closure of \textsl{ncc}-subgroups}\label{s6}
Let $\mathscr{X}$ be a flow with compact Hausdorff phase space $X$ and with phase mapping $T\times X\rightarrow X$, $(t,x)\mapsto tx$ in the sequel of this section.
\begin{se}\label{6.1}
Given a subgroup $H$ of $T$, let $H\curvearrowright X$ be the induced subflow of $\mathscr{X}$. We define
\begin{enumerate}
\item[] $R_H=\{(x,x^\prime)\,|\,x,x^\prime\in X\textrm{ s.t. }x^\prime\in \overline{Hx}\}$,
\end{enumerate}
which is called the \textit{$H$-orbital closure relation} of $\mathscr{X}$. Write
$X/H$ for $X/R_H$,
where $X/H$ is equipped with the quotient topology.
\end{se}

If $R_H$ is a closed $T$-invariant equivalence relation on $X$, then $X/H$ is a compact Hausdorff space (by Lem.~\ref{6.3} below) so that
$T\curvearrowright X/H$ is a flow. If $H\curvearrowright X$ is pointwise a.p., then $R_H$ is an equivalence relation. However, $R_H$ is generally not an invariant closed relation.

\begin{lem}\label{6.2}
If $\mathscr{X}$ is a minimal flow and $H\in\mathcal{S}^\textsl{ncc}(T)$, then $R_H$ is a closed invariant equivalence relation of $\mathscr{X}$.
\end{lem}

\begin{proof}
By Theorem~\ref{3.1}-(1) and (2), it follows that $R_H$ is a $T$-invariant equivalence relation on $X$. Further by Theorem~\ref{3.1}-(3), $R_H$ is closed in $X\times X$. The proof is completed.
\end{proof}

It turns out that, for $H\in\mathcal{S}^\textsl{ncc}(T)$, $X/R_H$ is a compact Hausdorff space by the following folklore lemma in general topology, where the ``compact Hausdorff'' of $X$ plays a role in its proof. We shall give a proof here for reader's convenience, which is comparable with Lemma~\ref{1.8}b.

\begin{lem}[{cf.~\cite[Lem.~4.9]{E69} without proof}]\label{6.3}
Let $R$ be an equivalence relation on $X$. Then $R$ is closed in $X\times X$ iff $X/R$ is a Hausdorff space.
\end{lem}

\begin{proof}
Let $\rho\colon X\rightarrow X/R$ be the quotient map.
If $X/R$ is Hausdorff, then $\Delta_{X/R}$ is closed in $X/R\times X/R$ so that $R=(\rho\times\rho)^{-1}[\Delta_{X/R}]$ is closed in $X\times X$. (This still holds for $X$ a topological space not necessarily compact and Hausdorff.)

Conversely, suppose $R$ is closed in $X\times X$.
Let $A\not=\emptyset$ be a closed subset of $X$ and set $R[A]=\bigcup_{a\in A}R(a)\subseteq X$. Clearly, $R[A]=\rho^{-1}\rho[A]$ is closed in $X$, for $R[A]$ is exactly equal to the image of $A\times X\cap R$ under the canonical map $(x_1,x_2)\mapsto x_2$.
So $\rho$ is closed. If $U\subseteq X$ is open, then 
\begin{enumerate}
\item[] $R^*(U):=\{x\in X\,|\,R(x)\subseteq U\}=X\setminus R[X\setminus U]=X\setminus\rho^{-1}\rho[X\setminus U]$
\end{enumerate}
is open. So $\rho[R^*(U)]$ is open in $X\setminus R$. Now if $\rho(x)\not=\rho(y)$ (or equivalently, $R(x)\cap R(y)=\emptyset$), then we can choose disjoint open sets $U_x$, $U_y$ with $R(x)\subseteq U_x$ and $R(y)\subseteq U_y$. Thus, $\rho[R^*(U_x)]\in\mathfrak{N}_{\rho(x)}(X/R)$ and $\rho[R^*(U_y)]\in\mathfrak{N}_{\rho(y)}(X/R)$ are disjoint, and, $X/R$ is a Hausdorff space. The proof is completed.
\end{proof}

\begin{rem}\label{6.4}
For the necessity proof above, ``compact Hausdorff'' of $X$ has played a role in proving that $R[A]$ is a closed subset of $X$.

If $R$ is closed in $X\times X$, then the cells $R(x)$, for $x\in X$, are closed sets in $X$. However, the converse is false (so that this negates \cite[A.2.3]{De}). Let's see a counterexample.
	
\begin{example}	
Let $X=[0,1]\times[0,1]=\{x=(x_1,x_2)\,|\,0\le x_1,x_2\le 1\}$. Define a relation $R$ on $X$ by the way: for $x=(x_1,x_2)$, $y=(y_1,y_2)\in X$,
\begin{enumerate}
\item[] $x\sim y$ iff $x_1=y_1$ if $0\le x_1,y_1<1$; $x\sim y$ iff $x=y$ if $x_1=y_1=1$.
\end{enumerate}
Clearly, $R$ is an equivalence relation on $X$ such that $R(x)$, for $x\in X$, is a closed set in $X$. However, $R$ is not closed in $X\times X$. Thus, $X/R$ is compact and not of Hausdorff.
\end{example}
\end{rem}

\begin{thm}\label{6.5}
Let $\mathscr{X}$ be a minimal flow and $H\in\mathcal{S}^\textsl{ncc}(T)$. Then for all $x\in X$, the evaluation mapping
\begin{enumerate}
\item[] $\mathfrak{e}\colon T/E_x[H]\rightarrow X/H,\quad tE_x[H]\mapsto t\overline{Hx}$
\end{enumerate}
is an isomorphism from $T\curvearrowright T/E_x[H]$ onto $T\curvearrowright X/H$, and moreover, $T\curvearrowright X/H$ is an a.p. factor of $\mathscr{X}$.
\end{thm}

\begin{note*}
Consequently, $T\curvearrowright X/H$ may be thought of as an a.p. coset flow of $T$ if $H\in\mathcal{S}^\textsl{ncc}(T)$. On the other hand, if $H\in\mathcal{S}^\textsl{sncc}(T)\setminus\mathcal{S}^\textsl{ncc}(T)$, then the statement is generally false. For example, let $H\lhd K\lhd G$ be as in Example~\ref{1.11}; let $X=G/H$ and $x=[e]=H\in X$. Then for $T=G$, $E_x[H]=H$ and by Example~\ref{3.4}A, $R_H$ is not closed so that $X/H$ is not Hausdorff by Lemma~\ref{6.3}. Since $G/E_x[H]=G/H$ is Hausdorff, hence $G/E_x[H]\not\cong X/H$.
\end{note*}

\begin{proof}
At first, since $\{\overline{Hy}\,|\,y\in X\}$ is a closed $T$-invariant partition of $X$ into $H$-minimal sets by Theorem~\ref{3.1}. So $X/H$ is a compact Hausdorff space and $T\curvearrowright X/H$ is a minimal flow by Lemmas~\ref{6.2} and \ref{6.3}. By Corollary~\ref{5.7}, $T\curvearrowright T/E_x[H]$ is an a.p. minimal flow. Obviously $\mathfrak{e}$ is a well-defined 1-1 onto map such that $\mathfrak{e}t=t\mathfrak{e}$ for all $t\in T$. It remains to show that $\mathfrak{e}$ is a continuous mapping. For this, consider the CD:
$$
\begin{diagram}
T&\rTo^{f\quad}&X/H\\
\dTo^\rho&\ruTo_{\mathfrak{e}}&\\
T/E_x[H]
\end{diagram}
\qquad \textrm{where}
\begin{cases}f\colon t\mapsto t\overline{Hx},\\\rho\colon t\mapsto tE_x[H].\end{cases}
$$
Notice that $t_n\to t$ in $T$ implies that $t_nx\to tx$; and $H\in\mathcal{S}^\textsl{ncc}(T)$ and $t_n\to t$ in $T$ implies that $f(t_n)\to f(t)$ in $X/H$. Thus, $f$ and $\rho$ both are continuous maps. Moreover, if $tE_x[H]=sE_x[H]$ (or equivalently, if $t^{-1}s\in E_x[H]$), then $f(t)=f(s)$ so that $\mathfrak{e}$ is continuous.
Thus, $T\curvearrowright X/H$ is equicontinuous. The proof is completed.
\end{proof}

\begin{cor}[{cf.~\cite[Thm.~2.25]{GH}}]\label{6.6}
Let $\mathscr{X}$ be a minimal flow, $x\in X$, and $S\in\mathcal{S}^\textsl{ns}(T)$. Then $\{\overline{Sz}\,|\,z\in X\}$ is a $T$-invariant partition of $X$; and moreover, $\{\overline{Sz}\,|\,z\in X\}$ and $T/E_x[S]$ have the same cardinality.
\end{cor}


\begin{cor}[{cf.~\cite[Lem.~2.8.26]{B79} for $T$ abelian}]\label{6.7}
Let $\mathscr{X}$ be a minimal flow, $x\in X$, and $A\in\mathcal{S}^\textsl{ns}(T)$. Then
$\mathfrak{e}\colon (T,T/E_x[A])\rightarrow(T,X/A)$
is an isomorphism.
\end{cor}


We now will conclude our arguments with an application of Theorem~\ref{2.3} and \ref{6.5}.

\begin{thm}\label{6.8}
Let $\mathscr{X}$ be a minimal flow, $H\in\mathcal{S}^\textsl{ncc}(T)$. Then $T\curvearrowright X/H$ is an a.p. nontrivial factor of $\mathscr{X}$ iff $T\curvearrowright X\times T/H$ is not a minimal flow.
\end{thm}

\begin{proof}
	If $T\curvearrowright X/H$ is nontrivial, then by Theorem~\ref{6.5}, $T\curvearrowright X/H$ is a common factor of $\mathscr{X}$ and $T\curvearrowright T/H$ so that $T\curvearrowright X\times T/H$ is not minimal. Now conversely, suppose $T\curvearrowright X\times T/H$ is not minimal. Then $H\curvearrowright X$ is not minimal by Theorem~\ref{2.3}. So $X/H$ is not a singleton. Now by Theorem~\ref{6.5}, $T\curvearrowright X/H$ is an a.p. nontrivial factor of $\mathscr{X}$. The proof is completed.
\end{proof}

\subsection*{Acknowledgements}
The author is grateful to the anonymous referee for her/his constructive comments.
The work was supported by National Natural Science Foundation of China (No.~12271245).

\bibliographystyle{plain}

\end{document}